\documentclass[12pt,twoside]{amsart}
\usepackage{amssymb}
\usepackage{verbatim}
\usepackage{amsmath}
\usepackage{bm}
\usepackage{a4wide}
\usepackage[latin1]{inputenc}
\usepackage[T1]{fontenc}
\usepackage{times}
\usepackage{amssymb,latexsym}
\usepackage{enumerate}
\usepackage[stable]{footmisc}
\usepackage{color}
\usepackage[colorlinks,linkcolor=black,citecolor=black,urlcolor=black]{hyperref}

\usepackage{ulem}

\DeclareMathOperator{\Dom}{Dom}
\DeclareMathOperator{\Ran}{Ran}
\DeclareMathOperator{\Id}{Id}
\newcommand{\db}{\bar\partial}
\newcommand{\HH}{\mathcal H}
\newcommand{\BB}{\mathcal B}
\newcommand{\C}{\mathbb C}
\newcommand{\R}{\mathbb R}
\newcommand{\Z}{\mathbb Z}
\newcommand{\dd}{\mathrm d}
\newcommand{\ii}{\mathrm i}

\newcommand{\bea}{\begin{eqnarray*}}
\newcommand{\eea}{\end{eqnarray*}}

\makeatletter
\newcommand{\sumprime}{\if@display\sideset{}{'}\sum%
            \else\sum'\fi}
\makeatother

\begin{document}

\numberwithin{equation}{section}

\newtheorem{theorem}{Theorem}[section]
\newtheorem{proposition}[theorem]{Proposition}
\newtheorem{conjecture}[theorem]{Conjecture}
\def\theconjecture{\unskip}
\newtheorem{corollary}[theorem]{Corollary}
\newtheorem{lemma}[theorem]{Lemma}
\newtheorem{observation}[theorem]{Observation}
\newtheorem{definition}{Definition}
\numberwithin{definition}{section} 
\newtheorem{remark}[theorem]{Remark}
\def\theremark{\unskip}
\newtheorem{kl}{Key Lemma}
\def\thekl{\unskip}
\newtheorem{question}[theorem]{Question}
\def\thequestion{\unskip}
\newtheorem{example}[theorem]{Example}
\def\theexample{\unskip}
\newtheorem{problem}{Problem}

\thanks{The first author was supported by National Natural Science Foundation of China,  No. 12271101. The  second author was supported by National Natural Science Foundation of China,  No. 12671104 and National Key R\&D Program of China, No. 2024YFA1015200.}

\address [Bo-Yong Chen] {Department of Mathematical Sciences, Fudan University, Shanghai, 200433, China}
\email{boychen@fudan.edu.cn}

\address  [Jujie Wu ]{School of Mathematics (Zhuhai), Sun Yat-Sen University, Zhuhai, Guangdong 519082, P. R.
 	China}
\email{wujj86@mail.sysu.edu.cn}

\title{Infinite-Dimensionality of Bergman Spaces on\ Complete $d$-Bounded K\"ahler Manifolds}
\author{Bo-Yong Chen  and Jujie Wu}

\date{}
\dedicatory{}

\begin{abstract}

Let $M$ be a complete K\"ahler manifold of complex dimension $n\geq 1$ whose K\"ahler form is $d$-bounded in the sense of Gromov. We prove that its Bergman space, namely the Hilbert space of $L^2$ holomorphic $n$-forms, is infinite-dimensional. The proof combines Gromov's construction of auxiliary holomorphic line bundles, parameter-uniform $\db$ estimates, and $L^2$ methods with singular weights.

  \bigskip

  
  \smallskip


\end{abstract}

\maketitle

\section{Introduction}

Let \((M,\omega)\) be a complete K\"ahler  manifold of complex dimension \(n\geq1\), and let \(L^2_{p,q}(M,\omega)\) denote the Hilbert space of square-integrable \((p,q)\)-forms with respect to the metric induced by \(\omega\) and the volume form \(dV_\omega=\omega^n/n!\). We take \(\bar\partial\) in its maximal closed realization and denote its Hilbert space adjoint by \(\bar\partial^*\). Set
\[
\mathcal H^{p,q}_{(2)}(M)=\mathcal H^{p,q}_{(2)}(M,\omega)
:=
\left\{
u\in\operatorname{Dom}\bar\partial\cap\operatorname{Dom}\bar\partial^*:
\bar\partial u=\bar\partial^*u=0
\right\}.
\]
Writing \(d=\partial+\bar\partial\), we define the Hodge and Dolbeault Laplacians by
\[
\Delta_d:=dd^*+d^*d,
\qquad
\Delta_{\bar\partial}
:=\bar\partial\bar\partial^*+\bar\partial^*\bar\partial.
\]
The K\"ahler identities give
$\Delta_d=2\Delta_{\bar\partial}$ on compactly supported smooth forms.
Since $(M,\omega)$ is complete, both Laplacians are essentially
self-adjoint on this domain, and their self-adjoint $L^2$ realizations
satisfy the same identity. Consequently,
$\mathcal H^{p,q}_{(2)}(M)$ coincides with the space of
$L^2$ Hodge harmonic forms of bidegree $(p,q)$.
On forms of type $(n,0)$, the target bidegree of $\db^*$ is $(n,-1)$, and hence $\db^*$ vanishes identically. Thus
\begin{equation}\label{eq:bergman}
 \HH^{n,0}_{(2)}(M)=\{u\in L^2_{n,0}(M):\db u=0\}.
\end{equation}
By local regularity, distributional holomorphicity implies holomorphicity in the usual sense. That is,
\begin{equation}\label{eq:bergman2}
 \HH^{n,0}_{(2)}(M) =\{u \in H^0(M,K_M): \|u\|_\omega <\infty\},
\end{equation}
where $K_M= \Lambda^{n,0}T^*M$ is the canonical bundle.
 Moreover, for $(n,0)$-forms,  the identity
\begin{eqnarray}\label{eq:invariance}
\|u\|_\omega^2
=
i^{\,n^2}\int_M u\wedge\overline u
\end{eqnarray}
shows that both square-integrability and the resulting Hilbert space structure are independent of the choice of \(\omega\). The intrinsic space \(\mathcal H^{n,0}_{(2)}(M)\) is usually called the {\it Bergman space} of \(M\).

\bigskip

A basic question is the following.
\begin{question}
Which complete K\"ahler manifolds have infinite-dimensional Bergman spaces?
\end{question}
Following Gromov~\cite{Gromov}, a complete K\"ahler metric $\omega$ is called {\it $d$-bounded} if there exists a globally defined smooth real one-form $\beta$ on $M$ such that
\begin{equation}\label{eq:dbounded}
 \omega=\dd\beta,\qquad \|\beta\|_{L^\infty(M,\omega)}<\infty.
\end{equation}
Here the norm is computed using the Riemannian metric associated with $\omega$. A compact K\"ahler manifold $(M,\omega)$ is called {\it K\"ahler hyperbolic} if the lift of $\omega$ to its universal cover $\widetilde M$ is $d$-bounded. Gromov proved that the Bergman space of the universal cover of a K\"ahler hyperbolic manifold is infinite-dimensional. More generally,
\[
 \dim_{\C}\HH^{p,q}_{(2)}(\widetilde M)=\infty,\qquad p+q=n.
\]
Starting from the $d$-bounded K\"ahler metric, Gromov introduced a family of connections $\{\nabla_t\}_{t\geq 0}$ on the trivial line bundle $\widetilde M\times\C$. These connections induce holomorphic line bundles $L_t$ with curvature $\ii F_t=t\omega$. He then obtained nonvanishing by combining his spectral gap estimates outside the middle degree with a twisted $L^2$ index theorem.

Without assuming a compact quotient, the first author claimed in~\cite{Chen} that the spaces of $L^2$ harmonic forms in all middle bidegrees are infinite-dimensional. That paper was subsequently retracted because of a gap in the proof of its main theorem. In the present paper, we give an independent proof for the bidegree $(n,0)$. We use Gromov's auxiliary holomorphic line bundles and compare the kernels for different parameters by means of parameter-uniform $\db$ estimates. We then construct arbitrarily many linearly independent sections using weights with logarithmic poles. The argument requires neither a compact quotient nor an $L^2$ index theorem.

\begin{theorem}\label{thm:main}
Let $(M,\omega)$ be a complete K\"ahler manifold of complex dimension $n\geq 1$. If $\omega$ is $d$-bounded, then the Bergman space of $M$ is infinite-dimensional.
\end{theorem}

The metric independence in (\ref{eq:invariance}) also gives the following obstruction to the existence of a complete $d$-bounded K\"ahler metric.
\begin{corollary}\label{cor:pluripolar}
Let $M=\widehat M\setminus E$, where $\widehat M$ is a compact complex manifold of dimension $n\geq 1$ and $E\subset\widehat M$ is a closed locally pluripolar set. Then $M$ admits no complete $d$-bounded K\"ahler metric. In particular, no positive-dimensional smooth quasi-projective manifold admits such a metric.
\end{corollary}
\begin{proof}
Suppose that $M$ admits a complete $d$-bounded K\"ahler metric. By Theorem~\ref{thm:main}, 
$$
\dim_{\C}\HH^{n,0}_{(2)}(M)=\infty.
$$
 Since closed locally pluripolar sets are removable for $L^2$ holomorphic functions, restriction gives an identification
\[
 \HH^{n,0}_{(2)}(M)\simeq H^0(\widehat M,K_{\widehat M}),
\]
where $K_{\widehat M}$ is the canonical bundle of $\widehat M$. The space on the right is finite-dimensional, a contradiction.
\end{proof}
\begin{question}
Under the hypotheses of Theorem~\ref{thm:main}, does an orthonormal basis of $\HH^{n,0}_{(2)}(M)$ define a holomorphic embedding into infinite-dimensional complex projective space?
\end{question}
By Kobayashi's classical paper~\cite{Kobayashi}, an affirmative answer would imply that the Bergman metric on $M$ exists.

The following question remains particularly challenging.
\begin{question}
Under the hypotheses of Theorem~\ref{thm:main}, is it true that
\[
 p+q=n\quad\Longrightarrow\quad\dim_{\C}\HH^{p,q}_{(2)}(M)=\infty?
\]
\end{question}
 
\section{A family of auxiliary holomorphic line bundles}\label{sec:hlb}


We introduce a family of connections $\{\nabla_t\}_{t\geq 0}$ on the trivial line bundle $L=M\times\C$. They induce holomorphic line bundles $L_t$ with curvature $\ii F_t=t\omega$. This construction is due to Gromov~\cite{Gromov}. To make it accessible to readers outside complex geometry, we provide a self-contained account.

The bundle $L$ has a fixed global frame
\[
 e_0(x)=(x,1).
\]
Every smooth section has a unique expression $s=fe_0$, where $f$ is an ordinary smooth complex-valued function. Thus one may initially regard a section as a complex-valued function accompanied by a fixed basis vector. A nowhere-vanishing section supplies a basis in each fiber and is called a {\it frame}. We fix the standard Hermitian metric by requiring
\[
 |e_0|=1,\qquad |fe_0|=|f|.
\]
This Hermitian metric specifies an inner product and a norm on each complex line fiber.

A connection is a rule for differentiating sections. The simplest such rule is
\[
 \dd(fe_0)=(\dd f)e_0.
\]
We modify it by setting
\begin{equation}\label{eq:connection}
 \nabla_t(fe_0)=(\dd f-\ii t f\beta)e_0,\qquad t\geq 0.
\end{equation}
Differentiation along a real tangent vector $X$ then gives
\[
 (\nabla_t)_X(fe_0)=\bigl(Xf-\ii t\beta(X)f\bigr)e_0.
\]
Thus $\nabla_t$ is the ordinary derivative plus a correction proportional to the function itself. It still satisfies the Leibniz rule
\[
 \nabla_t(as)=\dd a\otimes s+a\nabla_t s.
\]
In complex coordinates $z_1,\ldots,z_n$, the differential splits into two parts:
\[
 \dd f=
 \underbrace{\sum_j\frac{\partial f}{\partial z_j}\dd z_j}_{\partial f}
 +\underbrace{\sum_j\frac{\partial f}{\partial\bar z_j}\dd\bar z_j}_{\db f}.
\]
An ordinary function $f$ is holomorphic if and only if $\db f=0$. Similarly, a one-form decomposes as
\[
 \beta=\beta^{1,0}+\beta^{0,1},
\]
where $\beta^{0,1}$ is the part involving $\dd\bar z_j$. Taking the $(0,1)$ part of~\eqref{eq:connection}, we obtain
\begin{equation*}\label{eq:dbt}
 \db_t(fe_0)=(\db f-\ii t f\beta^{0,1})e_0.
\end{equation*}
We use this operator to define a new holomorphicity condition:
\begin{equation*}\label{eq:holomorphic}
 fe_0\text{ is a holomorphic section of }L_t
 \quad\Longleftrightarrow\quad\db f=\ii t f\beta^{0,1}.
\end{equation*}
When $t=0$, this is the ordinary Cauchy--Riemann equation. For $t>0$, a section $fe_0$ may be holomorphic with respect to $L_t$ without its coefficient $f$ being holomorphic in the usual sense.

We next explain why $L_t$ is a well-defined holomorphic line bundle. Seek a nowhere-vanishing local section $e_t=e^g e_0$ satisfying $\db_t e_t=0$. A direct calculation gives
\[
 \db_t(e^g e_0)=e^g(\db g-\ii t\beta^{0,1})e_0.
\]
It therefore suffices to solve locally
\begin{equation}\label{eq:g}
 \db g=\ii t\beta^{0,1}.
\end{equation}
The compatibility condition holds: since $\dd\beta=\omega$ and $\omega$ has type $(1,1)$,
\[
 \db\beta^{0,1}=(\dd\beta)^{0,2}=0.
\]
The local $\db$-Poincar\'e lemma states that a smooth $\db$-closed $(0,1)$-form is the $\db$ of a smooth function on a sufficiently small coordinate neighborhood. It therefore supplies a solution of~\eqref{eq:g}.

Once $e_t$ is chosen, every local section can be written as $s=ae_t$. The Leibniz rule gives
\[
 \db_t(ae_t)=(\db a)e_t+a\db_t e_t=(\db a)e_t.
\]
Consequently,
\begin{equation*}\label{eq:local-holomorphic}
 ae_t\text{ is holomorphic in }L_t
 \quad\Longleftrightarrow\quad a\text{ is an ordinary holomorphic function}.
\end{equation*}
Such an $e_t$ is called a {\it local holomorphic frame}: in this frame, holomorphicity of a section is equivalent to ordinary holomorphicity of its coefficient.

Suppose that on the intersection of two neighborhoods the frames satisfy
\begin{equation}\label{eq:transition}
 e_{t,j}=a_{ij}e_{t,i}.
\end{equation}
On a triple intersection,
\[
 a_{ik}e_{t,i}=e_{t,k}=a_{jk}e_{t,j}=a_{jk}a_{ij}e_{t,i},
\]
so $a_{ik}=a_{ij}a_{jk}$. Applying $\db_t$ to~\eqref{eq:transition} also yields $\db a_{ij}=0$. Thus the changes of frame are nowhere-vanishing holomorphic functions satisfying the cocycle condition, precisely the local gluing data for a holomorphic line bundle.

Finally, we compute the curvature of $L_t$. Relative to the fixed smooth unit frame $e_0$, let
\[
 e_t=e^g e_0,\qquad \db g=\ii t\beta^{0,1}
\]
be a local holomorphic frame. Write its Hermitian metric as $|e_t|^2=e^{-\varphi_t}$. Since $|e_0|=1$,
\[
 e^{-\varphi_t}=e^{g+\bar g},\qquad \varphi_t=-(g+\bar g).
\]
Because $\beta$ and $t$ are real, conjugating $\db g=\ii t\beta^{0,1}$ gives
\[
 \partial\bar g=-\ii t\beta^{1,0}.
\]
Using $\partial\db=-\db\partial$ and $\dd\beta=\omega$, we obtain
\begin{align*}
 \ii\partial\db\varphi_t
 &=-\ii\partial\db g-\ii\partial\db\bar g\\
 &=-\ii\partial(\db g)+\ii\db(\partial\bar g)\\
 &=t\partial\beta^{0,1}+t\db\beta^{1,0}\\
 &=t(\dd\beta)^{1,1}\\
 &=t\omega.
\end{align*}

\section{Proof of the main theorem}
We retain the notation $L_t$, $\nabla_t$, and $\db_t$ from Section~\ref{sec:hlb}. In particular,  $\ii F_t=t\omega$.  Using the fixed smooth unit frame, we identify the spaces for all parameters with
\[
 H^q=L^2\bigl(M,\Lambda^{n,q}T^*M\otimes L\bigr),\qquad L=M\times\C.
\]
All unmarked norms below are the corresponding $L^2$ norms.
Every $\db_t$ is given its maximal closed extension, namely
\[
 \Dom\db_t=\{u\in H^q:\db_tu\in H^{q+1}\text{ in the distributional sense}\}.
\]
Let $W$ denote exterior multiplication by $\beta^{0,1}$. Since $\beta$ is bounded, $W$ is bounded, and
\begin{equation}\label{eq:operator-difference}
 \db_t-\db_s=-\ii(t-s)W,
 \qquad
 \db_t^*-\db_s^*=\ii(t-s)W^*.
\end{equation}
In particular, the domains of $\db_t$ and of $\db_t^*$ are independent of $t$. On $H^1$, $\db_t^*$ is the adjoint of $\db_t:H^0\to H^1$, whereas $\db_t$ maps $H^1$ to $H^2$, with $H^2=\{0\}$ when $n=1$.

Put
\begin{equation*}\label{eq:kernel}
 \BB_t:=\ker(\db_t:H^0\longrightarrow H^1),
 \qquad
 P_t:H^0\longrightarrow\BB_t
\end{equation*}
for the closed holomorphic kernel and its orthogonal projection. Local regularity identifies $\BB_t$ with the $L^2$ holomorphic sections of $K_M\otimes L_t$; in particular, from (\ref{eq:bergman}) we have
\begin{equation*}\label{eq:kernel-zero}
 \BB_0=  \HH^{n,0}_{(2)}(M).
\end{equation*}

We recall the curvature estimate that will be used below. For an $L_t$-valued smooth compactly supported $(n,1)$-form $v$, the Bochner--Kodaira--Nakano identity gives
\begin{equation*}\label{eq:bkn}
 \|\db_t^*v\|^2+\|\db_tv\|^2
 \geq\bigl([\ii F_t,\Lambda_\omega]v,v\bigr)
 =t\|v\|^2.
\end{equation*}
Here $\Lambda_\omega$ is the adjoint of exterior multiplication by $\omega$. The last equality follows from
$[\omega\wedge,\Lambda_\omega]=(p+q-n)\Id$ on $(p,q)$-forms.

Completeness provides real cutoff functions $\chi_j\in C_c^\infty(M)$
such that $0\leq\chi_j\leq1$, $\chi_j\to1$ locally uniformly, and
$\|\dd\chi_j\|_{L^\infty}\to0$. Applying these cutoffs and then
local regularization shows that compactly supported smooth forms
are dense in 
\[
 \mathcal D:=\Dom\db_t^*\cap\Dom\db_t\subset H^1
\]
 for the norm
\[
 \bigl(\|v\|^2+\|\db_t^*v\|^2+\|\db_tv\|^2\bigr)^{1/2};
\]
 Consequently,
\begin{equation}\label{eq:bkn2}
 \|\db_t^*v\|^2+\|\db_tv\|^2
 \geq  t\|v\|^2, \quad v\in\mathcal D,\, t\geq0.
\end{equation}
see \cite[\S\S4--5]{Demailly}. This is the point at which completeness enters the global integration-by-parts argument.


\begin{proposition}\label{prop:uniform}
There exists a constant $C>0$, independent of\/ $t\geq0$, such that
\begin{equation}\label{eq:uniform}
 \|(\Id-P_t)u\|\leq C\|\db_tu\|,
 \qquad u\in\Dom(\db_t:H^0\longrightarrow H^1).
\end{equation}
\end{proposition}

\begin{proof}
For $v\in\mathcal D$, write
\[
 A_tv:=(\db_t^*v,\db_tv)\in H^0\oplus H^2.
\]
By \eqref{eq:operator-difference}, there is a constant $B>0$, depending only on the bound for $\beta$ and the dimension, such that
\begin{equation}\label{eq:perturbation}
 \begin{split}
 \|(\db_t-\db_s)u\|&\leq B|t-s|\,\|u\|,\qquad u\in H^0,\\
 \|(A_t-A_s)v\|&\leq B|t-s|\,\|v\|,\qquad v\in\mathcal D.
 \end{split}
\end{equation}
In the first line, the operator difference is understood as its bounded extension to $H^0$.
From (\ref{eq:bkn2}) we have
\begin{equation}\label{eq:bkn3}
 \|A_tv\|^2\geq t\|v\|^2,
 \qquad v\in\mathcal D,\, t\geq 0.
\end{equation}

Set
$
 \delta=\frac1{4B^2}.
$
Since the domain $\mathcal D$ is independent of the parameter,
\eqref{eq:perturbation} and \eqref{eq:bkn3} imply that, for every
$t\geq0$ and $v\in\mathcal D$,
\begin{equation*}\label{eq:gap}
 \begin{aligned}
 \|A_tv\|
 &\geq \|A_{t+\delta}v\|
       -\|(A_{t+\delta}-A_t)v\|\\
 &\geq \bigl(\sqrt{t+\delta}-B\delta\bigr)\|v\|\\
 &\geq \bigl(\sqrt{\delta}-B\delta\bigr)\|v\|
  =\frac1{4B}\|v\|.
 \end{aligned}
\end{equation*}
Thus
\begin{equation}\label{eq:uniform-coercivity}
 \|\db_t^*v\|^2+\|\db_tv\|^2
 \geq\frac1{16B^2}\|v\|^2,
 \qquad v\in\mathcal D,\, t\geq0.
\end{equation}

Fix $t\geq0$ and write
\[
 T=\db_t:H^0\longrightarrow H^1,
 \qquad
 S=\db_t:H^1\longrightarrow H^2.
\]
These operators are densely defined and closed. Set $\alpha=\beta^{0,1}$.   Since
$\db\alpha=(\dd\beta)^{0,2}=0$ and $\alpha\wedge\alpha=0$,
we have
\[
 \db_t^2u
 =-\ii t\,\db\alpha\wedge u
   -t^2\alpha\wedge\alpha\wedge u
 =0
\]
distributionally. By maximality of $S$, this gives
$\Ran T\subset\ker S$.
The standard $L^2$ existence lemma for closed Hilbert complexes (cf.  \cite[Theorem~1.1.2]{Hormander}),
applied to \eqref{eq:uniform-coercivity}, now implies that every
$f\in\ker S$ admits a solution $w\in\Dom T$ satisfying
\begin{equation}\label{eq:uniform-solution}
 \db_tw=f,
 \qquad
 \|w\|\leq4B\|f\|.
\end{equation}

Finally, let $u\in\Dom T$ and apply
\eqref{eq:uniform-solution} to $f=Tu$. Then $ \db_tu=  \db_tw$ and $\|w\| \leq 4B\| \db_tu\|$. Thus, $u-w\in\ker T=\BB_t$ and hence
\[
 (\Id-P_t)u=(\Id-P_t)w.
\]
The orthogonal projection $\Id-P_t$ is contractive, so
\[
 \|(\Id-P_t)u\|
 \leq\|w\|
 \leq4B\|\db_tu\|.
\]
This proves \eqref{eq:uniform} with $C=4B$. Moreover,
$(\Id-P_t)u$ is the unique solution of
$\db_tz=\db_tu$ orthogonal to $\BB_t$, and therefore the solution of
minimal $L^2$ norm.
\end{proof}

\begin{corollary}\label{cor:injection}
For every $k>0$, there is an injective linear map
\begin{equation}\label{eq:injection}
 \BB_k\longrightarrow\BB_0.
\end{equation}
\end{corollary}
\begin{proof}
By~\eqref{eq:uniform}--\eqref{eq:perturbation}, if $h\in\BB_t$, then
\begin{equation*}\label{eq:projection}
 \begin{aligned}
 \|(\Id-P_s)h\|&\leq C\|\db_sh\|
 =C\|(\db_s-\db_t)h\|\\
 &\leq CB|s-t|\,\|h\|.
 \end{aligned}
\end{equation*}
If $CB|s-t|<1$, the restriction $P_s|_{\BB_t}$ is injective. Subdividing any finite parameter interval into finitely many sufficiently short subintervals and composing these maps gives an injection between the kernel spaces at its endpoints.
\end{proof}

\begin{proposition}\label{prop:sections}
For every $N\in\Z_{>0}$, there exists $t=t_N>0$ such that
\begin{equation}\label{eq:dimension}
 \dim_{\C}\BB_t\geq N.
\end{equation}
\end{proposition}
\begin{proof}
Choose a relatively compact coordinate ball
\[
 U=\{z\in\C^n:|z|<r\}\Subset M
\]
such that $\omega\approx\ii\partial\db|z|^2$ on $U$.  According to the analysis in Section \ref{sec:hlb} we know that $L_t$ has a nowhere-vanishing holomorphic frame $e_t$ on a neighborhood of $\overline U$. In this frame, its Hermitian metric $h_t$ is represented by
\[
 |e_t|_{h_t}^2=e^{-\varphi_t},
\]
with $i\partial\bar\partial\varphi_t =t \omega$.
Choose a smooth cutoff function $\chi:\R\to[0,1]$ satisfying $\chi=1$ on $(-\infty,1/2]$ and $\chi=0$ on $[1,\infty)$. For $\epsilon\geq 0$, set
\[
 \phi_\epsilon=\chi(|z|^2/r^2)\log(|z|^2+\epsilon),\qquad
 \psi_{t,\epsilon}=\varphi_t+(n+N)\phi_\epsilon.
\]
Extend $\phi_\epsilon$ by zero outside $U$.  On \(\{|z|^2<r^2/2\}\), the function \(\phi_\epsilon\) is plurisubharmonic. On the fixed compact annulus
\[
A=\{r^2/2\leq |z|^2\leq r^2\}\Subset V,
\]
the function \(\log(|z|^2+\epsilon)\) and its derivatives up to order two are bounded uniformly in \(\epsilon\). Consequently, there exists a constant \(C_0\geq0\), independent of \(\epsilon\), such that
\[
\mathrm i\partial\bar\partial\phi_\epsilon\geq-C_0\omega
\quad\text{on }M.
\]
Choose
\[
t=t_N:=1+(n+N)C_0
\]
and keep this parameter fixed below. The smooth Hermitian metric
\[
h_{t,\epsilon}=h_t e^{-(n+N)\phi_\epsilon}
\]
has Chern curvature satisfying
\[
\mathrm iF_{t,\epsilon}
=
t\omega+(n+N)\mathrm i\partial\bar\partial\phi_\epsilon
\geq\omega.
\]
Equivalently, its local weight
\(\psi_{t,\epsilon}\) satisfies
\(\mathrm i\partial\bar\partial\psi_{t,\epsilon}\geq\omega\). 


Put
\[
 \sigma_j=z_1^j\chi(|z|^2/r^2)\,
 \dd z_1\wedge\cdots\wedge\dd z_n\otimes e_t,
 \qquad j=0,1,\ldots,N,
\]
and extend these sections by zero outside $U$. By~\cite[Theorem~5.1]{Demailly}, the equations
\[
 \db_tu_{j,\epsilon}=\db_t\sigma_j
\]
have solutions on $M$ satisfying
 
\[
\begin{aligned}
\int_M |u_{j,\epsilon}|_{\omega,h_t}^2
e^{-(n+N)\phi_\epsilon}\,dV_\omega
&\leq
\int_M |\db_t\sigma_j|_{\omega,h_t}^2
e^{-(n+N)\phi_\epsilon}\,dV_\omega\\
&\leq C,
\end{aligned}
\]
where $C$ is independent of $\epsilon$. The norm $|\cdot|_{\omega,h_t}$ is the norm on $L_t$-valued $(n,0)$-forms induced by $\omega$ and $h_t$.

Let $u_j$ be a weak $L^2$ limit along a sequence $\epsilon\to0$. Such a subsequence exists because $\phi_\epsilon$ is uniformly bounded above, so the weighted estimates imply a uniform unweighted $L^2$ bound. Then
\[
 \db_tu_j=\db_t\sigma_j
\]
in the sense of distributions, and
\[
 \int_M |u_j|_{\omega,h_t}^2
 e^{-(n+N)\phi_0}\,\dd V_\omega\leq C.
\]
The latter inequality follows by weak lower semicontinuity on compact subsets away from the center, followed by exhaustion.

Since $u_j$ is holomorphic near the center of the coordinate ball, this integrability condition forces all derivatives of its local coefficient of order at most $N$ to vanish there. Indeed, near the center the singular factor is $|z|^{-2(n+N)}$, which implies that 
\[
D^\alpha u_j(0)=0
\qquad\text{for every }|\alpha|\leq N.
\] 
Therefore
\[
 \sigma_j-u_j\in\BB_t,\qquad j=0,1,\ldots,N,
\]
and these sections are linearly independent, as their jets of order at most $N$ are the distinct monomials $1,z_1,\ldots,z_1^N$. Hence $\dim_{\C}\BB_t>N$.
\end{proof}

\begin{proof}[Proof of Theorem~\ref{thm:main}]
Combining~\eqref{eq:injection} with~\eqref{eq:dimension} yields
\[
 \dim_{\C}\BB_0\geq N.
\]
Since $N$ is arbitrary, $\HH^{n,0}_{(2)}(M)=\BB_0$ is infinite-dimensional.
\end{proof}

\bigskip

{\bf Acknowledgements.} The authors are grateful to  ChatGPT 6 Astra for valuable assistance with the exploration of some of the proof strategies and calculations. They  reviewed and verified the material they adopted and take final responsibility for all arguments and conclusions in the paper and for its adherence to standards of academic integrity.

\end{document}